\documentclass[11pt]{amsart}

\usepackage{amssymb,amsmath,amscd,graphicx,
latexsym,amsthm}
\usepackage{amssymb,latexsym,amsmath,amscd,graphicx}
\usepackage[mathscr]{eucal}
  \usepackage[all]{xy}
\def\OO{{\mathcal O}}

\def\F{\mathcal{F}}

\def\G{\mathcal{G}}
\def\H{\mathcal{H}}

\def\cP{\mathcal{P}}
\def\Pic0{{\rm Pic}^0}

\def\DD{{\mathbf{D}}}

\theoremstyle{plain}

\newtheorem{theorem}{Theorem}[section]
\newtheorem{theoremalpha}{Theorem}
\newtheorem{corollaryalpha}[theoremalpha]{Corollary}

\newtheorem{proposition/example}[theorem]{Proposition/Example}

\newtheorem{corollary}[theorem]{Corollary}
\newtheorem{lemma}[theorem]{Lemma}

\newtheorem{claim}[theorem]{Claim}

\newtheorem{step}[theorem]{Step}

\theoremstyle{definition}

\newtheorem{remark}[theorem]{Remark}
\newtheorem{example}[theorem]{Example}

\newtheorem{conjecture/question}[theorem]{Conjecture/Question}

\newtheorem{remark/definition}[theorem]{Remark/Definition}
\newtheorem{notation/assumptions}[theorem]{Assumptions/Notation}

\numberwithin{equation}{section}

 \theoremstyle{remark}

\begin{document}

\title{standard canonical support loci }

\author[G. Pareschi]{Giuseppe Pareschi}
\address{Dipartimento di Matematica, Universit\`a di Roma, Tor Vergata, V.le della
Ricerca Scientifica, I-00133 Roma, Italy} \email{{\tt
pareschi@mat.uniroma2.it}}
\dedicatory{Dedicated to Philippe Ellia on the occasion of his $60$th birthday}

\begin{abstract}  We consider the union of certain irreducible components of cohomological support loci of the canonical bundle, which we call standard.  We prove a structure theorem about them and single out some particular cases, recovering and improving results of Beauville and Chen-Jiang. Finally, as an example of application, we extend   to compact Kahler manifolds the classification of smooth complex projective varieties with $p_1(X)=1$, $p_3(X)=2$ and $q(X)=\dim X$. 
\end{abstract}

\thanks{}

\maketitle


\setlength{\parskip}{.1 in}

\section{introduction}

Let $X$ be a compact K\"ahler manifold. This paper is concerned with the cohomological support loci of the canonical bundle of $X$, namely
\begin{equation}\label{GLsets} V^i(K_X)=\{\mu\in \Pic0 X\>|\>h^i(K_X\otimes P_\eta)>0\>\}
\end{equation}
(here $P_\eta$ denotes the line bundle on $X$ corresponding to $\eta\in\Pic0 X$ via the choice of a Poincar\`e line bundle). 
One knows the following:

\noindent(a)  \emph{Every irreducible component  $W$ of   $V^i(K_X)$  is a translate of  a (compact) subtorus $T\subset \Pic0 X$. }\\
This can be rephrased as follows: we denote $\pi: \mathrm{Alb}\, X\rightarrow B:=\Pic0 T$ the dual quotient. This defines the composed map $f:X\rightarrow B$ sitting in the  commutative diagram
\begin{equation}\label{fundamental1}\xymatrix{X\ar[r]^{alb_X}\ar[rd]^f& \mathrm{Alb}\, X\ar[d]^\pi\\
&B\\}
\end{equation} 
Then $T=f^*\Pic0 B$  and, in this notation, item (a) is that, for some $\eta\in\Pic0X$,
\begin{equation}\label{compo}W=f^*\Pic0 B+ \eta \>  . 
\end{equation}

\noindent (b)  \emph{The map $f$
 verifies the inequality
\begin{equation}\label{fundamental2}
 \dim X-\dim f(X)\ge i \>.
\end{equation}}
 Both (a) and (b) are due to Green and Lazarsfeld (\cite{gl2}).

\noindent(c) \emph{The translating points $\eta$ are torsion modulo $f^*\Pic0 B$.  }\\
This is due to Simpson \cite{simpson} in the projective case (see also \cite{schnell-torsion}) and to Wang in the compact K\"ahler setting (\cite{wang}, see also  \cite{pps} \S12), proving a conjecture of Beauville and Catanese.

 The understanding of the loci $V^i(K_X)$ is often crucial in the study of irregular compact K\"ahler varieties but, despite the above powerful theorems, our knowledge of them is still unsatisfactory in some respects. 
   For example,  a clear geometric reason for the presence or absence of non-trivial $0$-dimensional components is lacking. Another issue is that
 there is no clear description of the points of finite order appearing as translating points. 
 In fact in \cite{schnell-torsion} 1.1.4-5 \ Schnell  shows an example where, for some $i$, $V^i(K_X)$  is not \emph{complete}, namely it happens that  $W=f^*\Pic0 B+\eta$   is a component of $V^i(K_X)$ but there  is a integer $k$ with $gcd(k,ord([\eta]))=1$ such that $f^*\Pic0 B+k\eta$ is not (in fact, it is not even contained in $V^i(K_X)$).

In this paper we show that there is a part of the cohomological support loci, which we call \emph{standard},  where the latter problem does not arise and moreover the translating points have a sufficiently clear geometric description. We will use the following terminology: a pair $(W,i)$, with $W$ irreducible component of $V^i(K_X)$, is called \emph{standard} if equality holds in (\ref{fundamental2}), i.e. 
\[\dim X-\dim f(X)=i.\]  
For example, the only  standard pair such that $W$ is $0$-dimensional  is $(\{\hat 0\},\dim X)$. 

The  union all subvarieties $W$ such that $(W,i)$ is a standard pair  will be referred to as \emph{the standard part of $V^i(K_X)$.}  As another immediate example, (\ref{fundamental2}) implies that the positive-dimensional part of $V^{\dim X -1}(K_X)$ coincides with  its standard part (of course they might be empty). 

According to some of the current literature, a morphism with connected fibers $g:X\rightarrow Y$ onto a normal compact analytic space $Y$ of maximal Albanese dimension (that is, the image of the Albanese map of $Y$ is equal to $\dim Y$) is called an \emph{irregular fibration}. We denote $\Pic0 (g)$ the kernel of the restriction map of $\Pic0 X$ to  a generic fiber of $g$. We recall that $\Pic0 (g)$ sits in an extension as
\begin{equation}\label{finite} 0\rightarrow g^*\Pic0 Y\rightarrow \Pic0 (g) \rightarrow \Gamma_g\rightarrow 0
\end{equation}
with $\Gamma_g$ a finite subgroup of $\Pic0 X/g^*\Pic0 Y$ \footnote{proof: let $X_g$ be a general fiber of $g$. Then $g$ induces a map $a:X_g\rightarrow K:=\ker(\mathrm{Alb}\, X\rightarrow \mathrm{Alb}\, Y)$ whose image spans $K$. Hence the homomorphism $\Pic0 K\rightarrow \Pic0 F$ has finite kernel. Then the exact sequence (\ref{finite}) follows from the dualization of the exact sequence $0\rightarrow K\rightarrow \mathrm{Alb}\, X\rightarrow \mathrm{Alb}\, Y\rightarrow 0$}.  Thus $\Pic0 (g)$ is the finite union of translated subtori $g^*\Pic0 Y+\eta$ for $\eta\in\Pic0 X$ such that its class modulo $g^*\Pic0 Y$ is  in $\Gamma_g$.

\begin{theoremalpha}\label{1} For each $i$ the standard part of $V^i(K_X)$ is the union, for all irregular fibrations $g:X\rightarrow Y$ with $\dim X-\dim Y=i$, of\\
- the subvarieties $\Pic0 (g) \smallsetminus (g^*\Pic0 (Y) +N_g)$, where $N_g$ is a finite subgroup of $\Pic0(g)$, if $\chi(K_{\widetilde Y})=0$, where $\widetilde Y$ is any  desingularization of $Y$  \footnote{the holomorphic Euler characteristic $\chi(K_{\widetilde Y})$ does not depend on the particular resolution $\widetilde Y$ considered. Since one can choose a $\widetilde Y$ which is Kahler  (see e.g. \cite{catanese2}, 1.9), $\chi(K_{\widetilde Y})\ge 0$ by generic vanishing (see below)}. If, in addition, the Albanese morphism of $Y$ is surjective then $N_g=\{0\}$. \\ 
 - the subgroups $\Pic0 (g)$ otherwise. 
\end{theoremalpha}

 It is perhaps worth to mention that Theorem \ref{1} can be also restated as the following Corollary. Given an irregular fibration $g:X\rightarrow Y$, let $i=\dim X-\dim Y$. By Koll\'ar's vanishing theorem (see Theorem \ref{decomp} below) $\Pic0 (g)$ is precisely the locus of  $\eta\in \Pic0 X$ such that $R^ig_*(K_X\otimes P_\eta)$ is non-zero (in fact a torsion-free sheaf of rank one). By deformation-invariance of holomorphic Euler characteristic $\chi(R^ig_*(K_X\otimes P_\eta))$ depends only on $[\eta]\in \Pic0 (g)/g^*\Pic0 Y$. By generic vanishing (see Theorem \ref{vanishing} below) it is always non-negative.

\begin{corollaryalpha}\label{2mezzo}  In the notation above, the set of $[\eta]\in\Pic0 (g)/g^*\Pic0 Y:=\Gamma_g$ such that \break $\chi(R^ig_*(K_X\otimes P_\eta))=0$ is either empty or a subgroup of $\Gamma_g$. The latter case holds if and only if  $\chi(K_{\widetilde Y})=0$, where $\widetilde Y$ is any desingularization of $Y$. If this is the case and, in addition, the Albanese map of $Y$ is surjective then such subgroup is zero.
\end{corollaryalpha}

In the following particular case, we obtain a more precise description, recovering and extending a well known result of Beauville (\cite{beauville}, Cor.2.3).

\begin{theoremalpha}\label{2} Let $p$ be the maximal index such that $V^p(K_X)$ is positive-dimensional \footnote{if $V^i(K_X)$ is empty or zero-dimensional for all $i$ we define $p=-\infty$}\\
 (a) If $g:X\rightarrow Y$ is an irregular fibration with $\dim X-\dim Y>p$ then
 $Y$ is bimeromorphic to a complex torus and $\Pic0(g)=g^*\Pic0 Y$. \\
 (b)  The positive-dimensional part of $V^{p}(K_X)$ is the union, for all irregular fibrations $g:X\rightarrow Y$ with  $\dim X-\dim Y=p$, of:\\
- the subvarieties $\Pic0(g)\smallsetminus g^*\Pic0(Y)$  if  $Y$ is  bimeromorphic to a complex torus,\\
-  the subgroups $\Pic0(g)$ otherwise.
\end{theoremalpha}
Beauville's aforementioned result  is  statement  (b)  for $V^{\dim X-1}(K_X)$ (note that 
 $p=\dim X-1$ if $V^{\dim X-1}(K_X)$ is positive-dimensional). In this case the positive-dimensional components are induced by fibrations onto smooth curves of genus $\ge 1$. Even in this case our proof is different from Beauville's.
There some questions  naturally connected to the above results:

\noindent (a) 
 It would be interesting to give a geometric description for the groups $\Pic0(g)$ associated to irregular fibrations (and of the subgroups $N_g$ appearing in the statement of Theorem \ref{1}). When the base $Y$ is a curve, $\Pic0(g)$ is completely described in terms of the multiple fibers of $g$ (Beauville \cite{beauville}). \\
 (b) Topological invariance of irregular fibrations of compact Kahler manifolds  and their number (up to equivalence), as well as their relation with the fundamental group. This matter is well understood when the basis of the fibration
is a curve of genus $\ge 2$ by work of Siu, Beauville and Arapura (see \cite{catanese1}, \cite{beauville}, \cite{arapura}, \cite{simpson}). In Beauville and Arapura's treatment the main ingredient is the above mentioned result of Beauville. In view of Theorem \ref{2}(b) it is natural to ask for similar results  for irregular fibrations over normal analytic spaces $X\rightarrow Y$ of arbitrary dimension, at least when $\chi(K_{\widetilde Y})>0$. Interestingly,   Catanese (\cite{catanese1}) proved -- with a different approach -- the topological invariance of the existence of irregular fibrations $X\rightarrow Y$ such that the Albanese map of $Y$ is non-surjective  (note that when $Y$ is a smooth curve the condition $\chi(K_Y)>0$, i.e. $g(Y)\ge 2$, is equivalent to the non-surjectivity of the Albanese map of $Y$, but this is not anymore  true in higher dimension). \\
 (c) A conjecture of M. Popa predicts that all loci $V^i(K_X)$ are derived-invariants.  In view of Theorem \ref{2}         this would imply  that the integer $p$ is a derived invariant, as well as all (equivalence classes of)  irregular fibrations $g:X\rightarrow Y$ such that $\dim X-\dim Y=p$    (except those such that $Y$ is bimeromorphic to a complex torus and $\Pic0(g)=g^*\Pic0 Y$). Again, this is (partly) known when the base $Y$ is a smooth curve (\cite{popa}, \cite{popa2}). \\
(d) It would be interesting to find a larger part of canonical cohomological suppport loci admitting a description similar to Theorem \ref{1}.

Another consequence of Theorem \ref{1} is

\begin{corollaryalpha}\label{3}  Assume that  $X$ has maximal Albanese dimension (that is: $\dim alb_X(X)=\dim X$) and that $V^0(K_X)$ is a proper subvariety of $\Pic0 X$. Then
$V^0(K_X)$ is the union, 
for all irregular fibrations $g:X\rightarrow Y$ such that $\dim X-\dim Y=\dim \mathrm{Alb}\, X-\dim \mathrm{Alb}\, Y$, of:\\
- the subvarieties $\Pic0 (g)\smallsetminus(g^*\Pic0 (Y) +N_g)$, where $N_g$ is a finite subgroup of $\Pic0(g)$, if $\chi(K_{\widetilde Y})=0$ (where $\widetilde Y$ is any desingularization of $Y$). If, in addition, the Albanese morphism of $Y$ is surjective then $N_g=\{0\}$.\\ 
 - the subgroups $\Pic0 (g)$ otherwise. 
\end{corollaryalpha}

A weaker statement along these lines was proved in \cite{pareschi}, 4.3 . Corollary \ref{3} is also a strengthening of a result of  \cite{chen-jiang}, (Th. 3.5) (in turn generalized in \cite{pps} Cor. 16.2), asserting that   $V^0(K_X)$ is invariant with respect to the natural involution of $\Pic0 X$. 
Note that from Theorems \ref{1} and \ref{2}, and  Corollary \ref{3} it follows that the loci in question are \emph{complete} in the above sense.

All proofs are based on Hacon's generic vanishing theorem for higher direct images, often combined with Koll\'ar's decomposition for the derived direct image of the canonical sheaf.  A key tool for the proof of Theorem \ref{1} is a sharper version of Hacon's theorem introduced by J. A. Chen and Z. Jiang. We will be refer  to that as the Chen-Jiang decomposition.  Hacon's and Chen-Jiang's theorem was extended    to  the compact  K\"ahler setting and  to higher direct images in \cite{pps}. 

Results as the above are  useful in applications concerning the geometry and  the classification of irregular compact K\"ahler manifolds. In what follows we will denote $p_i(K_X)=h^0(K_X^i)$ the plurigenera of a compact K\"ahler manifold $X$.  For example, already the aforementioned invariance of $V^0(K_X)$ under the natural involution of $\Pic0 X$  was a key point in the proof that complex tori are classified by their irregularity and first two plurigenera (\cite{pps}, Th.B) (this was  a theorem of Chen-Hacon in the algebraic case (\cite{chen-hacon1}). 
In the last section, which is somewhat independent, we provide a related  application. In fact, after  complex tori, it is natural to aim at the classification of compact K\"ahler irregular manifolds with  low plurigenera. In the projective case  this has been pursued by various authors, see  \cite{chen-hacon2}, \cite{chen-hacon3}. \cite{hacon2}, \cite{hacon-pardini}, \cite{jiang}. Still under the conditions $q(X)=\dim X$ and $p_1(X)\ne 0$, it turns out that the next   lowest  condition on plurigenera is $p_3(X)=2$. 
We confirm the classification of such varieties in the projective case, due to Hacon-Pardini (\cite{hacon-pardini} Th.4)\footnote{strictly speaking  the proof of this application uses only the well known theorem of Beauville mentioned above, which is now a particular case of Theorem \ref{2}. However we included it in this paper because it is suggestive about the possible use of Theorems \ref{1} and \ref{2} when dealing with this sort of problems}.

\begin{theoremalpha}\label{4} Let $X$ be a compact Kahler manifold with $q(X)=\dim X$, $p_1(X)\ne 0$ and  $p_3(X)=2$. Then $\mathrm{Alb}\, X$ has a quotient  (with connected fibers) $\pi: X\rightarrow E$ with  $E$ elliptic curve, and $X$ is bimeromorphic to the
ramified double cover of $a:X\rightarrow \mathrm{Alb}\, \, X$ such that 
\[a_*K_X= \OO_{\mathrm{Alb}\, X}\oplus (\pi^*\OO_E(p)\otimes P_\eta),\] where $p$ is a point of $E$ and $\eta$ is an element of order two of $\Pic0 X\smallsetminus\pi^*\Pic0 E$. 
\end{theoremalpha}

 It should be mentioned that the result of \cite{hacon-pardini} is stronger, since it works without the hypothesis $p_1(X)\ne 0$, which, in our treatment, is used to ensure the surjectivity of the Albanese map. I will come back to this point in the future.
However, apart from this issue, the argument here seems to be simpler and more self-contained. Hopefully this method will find more application to the classification of irregular compact K\"ahler manifolds. 

The paper is organized as follows: there are five background sections, containing  material probably known to the experts, but not entirely found in the literature. The reader can use them as a glossary, starting directly from \S6. Although Theorem \ref{2} is essentially a more precise version of a particular case Theorem \ref{1}, as a matter of expository preference we prove it directly in \S6, with a  simpler and more self-contained argument. Theorem \ref{1} and the other corollaries are proved in \S7. Theorem \ref{4} is proved in \S8.

\medskip\noindent\textbf{Acknowledgement. }  I am very indebted to Zhi Jiang for pointing out  some gaps and errors in a previous version of this paper. 

\section{Background material: GV, M-regular, extremal components}

Let $X$ be a compact K\"ahler manifold and $a:X\rightarrow A$ a morphism to a complex torus. Given a coherent sheaf $\F$ on $X$ one can consider its cohomological support loci (with respect to $a$)
\begin{equation}\label{GLsetsII}V^i_a(X, \F)=\{\alpha\in \Pic0 A\>| \> h^i(X,\F\otimes a^* P_\alpha)>0 \}
\end{equation}
As for cohomology groups, we will often suppress $X$ from the notation, writing simply $V^i_a(\F)$. The cohomological support loci defined in (\ref{GLsets}) are  particular cases of the above, since $\Pic0 X=alb^*\Pic0 (\mathrm{Alb}\, X)$ via the Albanese map \ $alb:X\rightarrow \mathrm{Alb}\, X$. If this is not cause of confusion, when $a=alb$ we suppress it from the notation, and simply write $V^i(\F)$.  

A coherent sheaf $\F$ on $X$ is said to be $GV$ (with respect to $a$) if 
\[\mathrm{codim}_{\Pic0 A}V^i_a(\F)\ge i\quad\hbox{ for all  $i$.}\]
 In particular, this implies that the loci $V^i_a(\F)$ are strictly contained in $\Pic0 A$ for all $i> 0$.  Hence $\chi(\F)\ge 0$, and $\chi(\F)>0$ if and only if $V^0_a(\F)=\Pic0 A$.  

A GV sheaf $\F$ (with respect to $a$) is said to be \emph{$M$-regular} if the inequalities are strict for $i>0$, namely  $\mathrm{codim}_{\Pic0 A}V^i_a(\F)> i$ for  $i>0$. Therefore, the difference between GV and M-regular is the presence of subvarieties of $V^i$ of codimension $i$ for some $i>0$. Henceforth we  will refer at them as \emph{extremal components}. This difference is best appreciated via the Fourier-Mukai transform associated to a Poincar\`e line bundle.
We refer to the surveys \cite{pp3} and \cite{pareschi} or to the papers \cite{chen-jiang}, \cite{pps} for a thorough discussion of this aspect. Here we will give just a minimal account. 

In the rest of this section we will assume that $\F$ is coherent sheaf on the complex torus $A$ (and the morphism $a$ is simply the identity). Let  $\cP$ be a Poincar\`e line bundle on  $A\times \Pic0 A$. Let
\[\Phi_{\cP}:\DD(coh(\OO_A))\rightarrow \DD(coh(\OO_{\Pic0 A}))\]
be the Fourier-Mukai functor associated to $\cP$.  As it is well known, this is an equivalence  (\cite{mukai}, see also \cite{huybrechts}, and \cite{bbbp} for the case of non-algebraic complex tori). We consider also the (unshifted)  dualizing functor 
\ $\DD(coh(\OO_A))\rightarrow \DD(coh(\OO_{A}))$ \ 
defined by $\F^\vee=\mathbf{R}\mathcal{H}om(\F,\OO_A)$. 
\begin{theorem}\label{FM} Let $q=\dim A$. Let $\F$ be a coherent sheaf on a complex torus $A$. 

\noindent (a) $\F$ is a GV sheaf  if and only if $\Phi_{\cP}(\F^\vee)$ is a sheaf in cohomological degree $q$, denoted $\widehat{\F^\vee}[-q]$. The support of $\widehat{\F^\vee}[-q]$ is $-V^0(\F)$. 

\noindent (b) A GV sheaf $\F$ is M-regular if and only if the sheaf $\widehat{\F^\vee}$  torsion-free. If $\F$ is not M-regular,  to each extremal component for $\F$, say $W$, corresponds  a torsion subsheaf of  $\widehat{\F^\vee}$ (supported on $-W$) and conversely.
\end{theorem}
As mentioned above, this  a well known fact, see e.g.  \cite{pp3} and \cite{pareschi}.
A consequence of Theorem \ref{FM} is the following non-vanishing result, see e.g. \cite{pareschi} Lemma 1.12 where it is stated only in the algebraic case, but its proof works for complex tori as well. 

\begin{corollary}\label{non-vanishing}  Let  $\F$ be a non-zero  GV sheaf on a complex torus $A$.  Then\\
(a) $V^0(\F)\ne\emptyset$.  \\
 (b) if $\F$ is a M-regular  then $V^0(\F)=\Pic0 A$ and $\chi(\F)>0$.
 \end{corollary}

 Here is a basic  example of GV but non-M-regular sheaf occurring frequently in what follows.

\begin{example}\label{pullback}[Pullback of M-regular sheaves  on quotients] Let $\pi:A\rightarrow B$ be a surjective morphism of complex tori, with $\dim A-\dim B=m>0$, and let $\F$ be a $M$-regular sheaf on $B$. Then for all  $\alpha\in \Pic0 A$ the sheaf (on $A$) $ \pi^*(\F)\otimes P_\alpha^{-1}$ is  GV  but  not M-regular.  Indeed,   supposing for simplicity that $\pi$ has connected fibers, for $j\le  m$ we have that
\begin{equation}\label{banale} R^j\pi_*(\pi^*\F\otimes P_\gamma)=\begin{cases} (\F\otimes P_\beta)^{\oplus {m\choose j}}&\hbox{for $P_\gamma=\pi^*P_\beta$ with $\beta \in\Pic0 B$}\\0&\hbox{otherwise}\\
\end{cases}
\end{equation}
Since $\F$ is assumed to be M-regular we have that
\begin{equation}\label{banale2} V^k(B,\F)=\begin{cases}\Pic0 B&\hbox{for $k=0$}\\
\subsetneq\Pic0 B&\hbox{otherwise.}\\
\end{cases}
\end{equation}
  Therefore, combining (\ref{banale}), (\ref{banale2}), projection formula and the Leray spectral sequence we get that 
  \[V^k(A,pi^*(\F)\otimes P_\alpha^{-1})=\pi^*\Pic0 B+\alpha\quad\hbox{for $k=0,..,m$}\] 
for all $\alpha\in\Pic0 A$.  In particular
$V^m(A, \pi^*(F)\otimes P_\alpha^{-1})$  has codimension $m$, hence $\pi^*F\otimes P_\alpha^{-1}$ is GV but it is not M-regular. 
A similar computation shows that $\pi^*\Pic0 B+\alpha$ is the only extremal component for the sheaf $\pi^*(\F)\otimes P_\alpha^{-1}$. 

This is perhaps more  suggestively seen from the Fourier-Mukai point of view.  Here  we will  use the following basic fact about Fourier-Mukai transforms associated to Poincar\`e line bundle on complex tori. Let $\pi:A\rightarrow B$ be a quotient of complex tori,  and let 
 $\hat\pi:\Pic0 B\rightarrow \Pic0 A$ be the dual homomorphism. Then we have the following natural isomorphism of functors (see \cite{chen-jiang} Prop. 2.3, where it is stated for abelian varieties, but the proof works for complex tori without changes)
 \begin{equation}\label{identi} \Phi_{\cP_A}\circ\pi^* \cong \hat\pi_*\circ \Phi_{\cP_B}[\dim B-\dim A]
 \end{equation}
 Gong back to the subject of the present Example, we know from Theorem \ref{FM}(b) that the Fourier-Mukai trasform on $B$  of $F^\vee$ is a torsion-free sheaf, say $\G$, in cohomological degree $\dim B$ on $\Pic0 B$. By (\ref{identi}) the Fourier-Mukai trasform on $A$ of $(\pi^*\F)^\vee$ is the torsion sheaf, in cohomological degree $\dim A$, consisting of the torsion-free sheaf $\G$ on $\pi^*\Pic0 B$ seen as a torsion sheaf on $\Pic0 A$. 
 
 Similarly the Fourier-Mukai transform on $A$ of $\pi^*(\F\otimes P_\alpha^{-1})^\vee$ is the torsion-free sheaf on $\pi^*\Pic0 B-[\alpha]=-V_0(\pi^*(\F)\otimes P_{\alpha^{-1}})$, seen as a torsion sheaf on $\Pic0 A$. 
\end{example}

 \section{Background material: Generic vanishing theorem and Chen-Jiang decomposition, I}
 
 The idea of generic vanishing  is due to Green and Lazarsfeld (\cite{gl1} and \cite{gl2}). Since then their theorems have been extended in various directions. One of these is generic vanishing for higher direct images of canonical sheaves, due to Hacon, \cite{hacon}. The most updated version is Theorem \ref{vanishing} below. 
  The idea of a decomposition as in the statement  was introduced  for $i=0$ in the projective case in \cite{chen-jiang}. This was extended to all $i$, also in the compact K\"ahler setting, was proved in \cite{pps} using the theory of Hodge modules.  It will be referred to as  \emph{Chen-Jiang decomposition}.

\begin{theorem}\label{vanishing} \emph{(\cite{hacon}, \cite{chen-jiang}, \cite{pps})} Let $f:X\rightarrow A$ be a morphism from a compact K\"ahler manifold to a complex torus.  Let $\eta$ be a point of finite order of $\Pic0 X$. Then, for all $i$, 
\[R^if_*(K_X\otimes P_\eta)=\bigoplus_k \pi_k^*(\F_k)\otimes P_{\alpha_k}^{-1}\]
where:  each $\pi_k:A\rightarrow B_k$ is a surjective morphism of complex tori with connected fibers,  $\F_k$ is a M-regular coherent sheaf supported on a projective subvariety of  $B_k$, 
 and $\alpha_k$ is a point of finite order of $\Pic0 A$. 
In particular, $R^if_*(K_X\otimes P_\eta)$ is a 
GV sheaf on $A$.
\end{theorem}

\begin{remark}\label{important}(Chen-Jiang summands.)
(a) Note that the homomorphisms $\pi_k$ can include the identity of $A$. By Corollary \ref{non-vanishing}, this happens when $\chi(R^if_*(K_X\otimes P_\eta))>0$, which is  the generic rank of the M-regular summand. Since the support of $R^if_*(K_X\otimes P_\eta)$ is torsion-free   on $f(X)$ (Theorem \ref{decomp} below)  by Theorem \ref{vanishing} $f(X)$ is a projective variety in this case.

\noindent  (b)  By Example \ref{pullback}  other summands in the Chen-Jiang decomposition appear if and only if   $R^if_*(K_X\otimes P_\eta)$ is not M-regular. More precisely: \emph{there is exactly one of them for each pair $(m,W)$ with $m>0$ and $W$ a  extremal component of $V^m(R^if_*(K_X\otimes P_\eta))$} (here $m=\dim A-\dim B_k$).
\end{remark}

The important results summarized in the following Corollary are originally due to Ein and Lazarsfeld (\cite{el}). In the present treatment they follow at once from the Chen-Jiang decomposition and Example \ref{pullback}

\begin{corollary}\label{el}\emph{(Ein-Lazarsfeld)}
   (a) Assume that $V^0(R^if_*(K_X\otimes P_\eta))$ is a proper subvariety of $\Pic0 X$ \emph{( i.e. $\chi(R^if_*(K_X\otimes P_\eta))=0$)}.  Then, for each $j>0$, every component of codimension $j$ of $V^0(R^if_*(K_X\otimes P_\eta))$ is also an extremal component, namely a $j$-codimensional component  of $V^j(R^if_*(K_X\otimes P_\eta))$. In particular, if there is an isolated point in $V^0(R^if_*(K_X\otimes P_\eta))$ then $V^{\dim A}(R^if_*(K_X\otimes P_\eta))$ is not empty, hence the map $f$ is surjective.

 \noindent (b)  Extremal components  of $R^if_*(K_X\otimes P_\eta)$ are subtori-translates of the form 
 $\pi_k^*(\Pic0 B_k)+\alpha_k$ such that the fibers of the map $\pi_k\circ f:X\rightarrow B_k$ surject on the fibers of the homorphism $\pi_k:A\rightarrow B_k$. Equivalently: $\dim f(X)-\dim \pi_k(f(X))=\dim A-\dim B_k$. 
 \end{corollary}

The Fourier-Mukai meaning of the Chen-Jiang decomposition is summarized in the following

\begin{remark}\label{CJFM}(FM transform and Chen-Jiang decomposition.) Let $i\le \dim X-\dim f(X)$, and  denote $\mathcal{R}_i=R^if_*(K_X\otimes P_\eta)$. Assuming that it is non-zero, the combination of Theorems  \ref{vanishing} and \ref{FM} tells that:\\
(i) the FM-transform of $\mathcal{R}_i^\vee$ is a sheaf in cohomological degree $\dim A$: \  $\widehat{\mathcal{R}_i^\vee}[-\dim A]$ ;\\
(ii) the sheaf $\widehat{\mathcal{R}_i^\vee}$ is the direct sum of its torsion part and its torsion-free part (one of them can be zero);\\
(iii)  the torsion part of  $\widehat{\mathcal{R}_i^\vee}$  (if any) is the direct sum of torsion-free sheaves on translates of subtori $\pi_k^*\Pic0 B_k-[\alpha_k]$, seen as sheaves on $\Pic0 A$. These sheaves are the translates by $-[\alpha_k]$ of the transforms on $B_k$ of $\F_k^\vee$. They are in 1-1 correspondence with the extremal components. 
\end{remark}

\begin{remark}\label{uniqueness} (Uniqueness of the Chen-Jiang decomposition.) From the previous Remark it follows that  the Chen-Jiang decomposition is essentially \emph{canonical}: the sheaves $\pi_k^*\F_k\otimes P_{\alpha_k}^{-1}$ are essentially unique. In fact  -- via the inverse FM functor  \ $\DD(coh(\OO_{\Pic0 A}))\rightarrow \DD(coh(\OO_{A}))$ -- their duals are the transforms respectively of the torsion-free part of $\widehat{\F^\vee}$ and of the components of the torsion part of $\widehat{\F^\vee}$. In particular, their supports, namely the translated subtori $\pi_k^*\Pic0 B_k+\alpha_k$ are uniquely determined (up to reordering)\footnote{ the surjective homorphisms with connected fibres $\pi_k:A\rightarrow B_k$ are not uniquely determined. However one can arrange them in such a way that $\pi_k$ factorizes trough $\pi_h$ if $\pi_k^*\Pic0 B_k$ is contained $\pi_h^*\Pic0 B_h$}.
\end{remark}

 \section{Background material: Generic vanishing theorem and Chen-Jiang decomposition, II}

\subsection{Koll\'ar decomposition} 

This is the other essential tool. We state it only in the version we will need 

\begin{theorem}\label{decomp}
Let $f: X \rightarrow Y$ be a proper morphism from a compact K\"ahler
manifold to a reduced and irreducible analytic space, and let $\eta$ be a torsion point of $ \Pic0(X)$. Then, in the derived category $\DD(coh(\OO_Y))$,
 \[{\bf R} f_*(K_X\otimes P_\eta) = \bigoplus_j \bigl( R^j f_* (K_X\otimes P_\eta) \bigr) [-j]\]
Moreover, if $f$ is surjective, then $R^j f_*(K_X \otimes P_\eta)$ is torsion-free for every
$j\ge 0$. In particular, it vanishes for $j > \dim X - \dim Y$. In general $R^j f_*(K_X \otimes P_\eta)$ vanishes for $j>\dim X-\dim f(X)$.
\end{theorem}

This theorem is due to Koll\'ar in the case when $Y$ is projective. When $Y$ is an analytic space the degeneration of the Leray spectral sequence at $E_2$ and the torsion-freeness 
 are due to Takegoshi \cite{takegoshi}.  Saito \cite{saito2} greatly generalized the results of Koll\'ar, using the
theory of Hodge modules.  Using \cite{saito1} his treatment works also in the analytic setting, as stated in Theorem \ref{decomp} (see also \cite{pps}).  

A standard consequence, proved in \cite{kollar2} Thm. 3.4 (in the algebraic case, however the same proof goes over in the analytic setting) is that the previous statement is still valid replacing the pair $(X,K_X)$ with the pair $(Y,R^jf_*(K_X\otimes P_\eta))$, for any $j\le \dim X-\dim Y$:

\begin{corollary} \label{higher} Let $X$ be a compact K\"ahler manifold and let $Y$,$Z$ be reduced and irreducible analytic spaces. Let $f:X\rightarrow Y$ and $a: Y\rightarrow Z$ proper surjective morphisms. Let $\eta\in\Pic0 X$ and $\beta\in \Pic0 Y$ be torsion line bundles.  Then for all $i$:

\noindent (a) \[R^i(a\circ f)_*(K_X\otimes P_\eta)=\bigoplus_{j}R^ja_*R^{i-j}f_*(K_X\otimes P_\eta)\]

\noindent (b) \[
	{\bf R} a_*(R^if_*(K_X\otimes P_\eta)\otimes P_\beta) \simeq \bigoplus_{j} \bigl( R^j a_* (R^if_*(K_X\otimes P_\eta) \otimes P_\beta)\bigr) [-j]
\]
in the derived category $\DD(coh(\OO_Z))$.

\noindent (c) For all $i$ and $j$ the sheaf  $R^ja_*(R^i f_*(K_X \otimes P_\eta)\otimes P_\beta)$ is torsion-free. In particular, it vanishes for $j > \dim Y - \dim Z$.
\end{corollary}

The proof is as Thm. 3.4 of \emph{loc cit} (note
that, under the hypotheses of the Theorem, $P_\eta\otimes f^*P_\beta$ is a torsion line bundle.). Note also that item (iv) of \emph{loc cit}, which makes sense only in the projective case, is not used to prove the other assertions. 
Combining Theorem \ref{vanishing} with Corollary \ref{higher} we obtain 
\begin{theorem}\label{casino} Let $X$ be a compact Kahler manifold, $f:X\rightarrow A$ a morphism to a complex torus, and $\pi: A\rightarrow B$ a homomorphism of complex tori. Let also $\eta\in \Pic0 X$ and $\beta\in \Pic0 B$ be points of finite order. Then, for each $i$ and $j$
\[R^j\pi_*(R^if_*(K_X\otimes P_\eta)\otimes P_\beta)=
\bigoplus_k \sigma_k^*(\G_k)\otimes P_{\gamma_k}^{-1}\]
 where: $\sigma_k:B\rightarrow C_k$ is a surjective morphism of complex tori with connected fibers, each $\G_k$ is a M-regular coherent sheaf supported on a projective subvariety of the complex torus $C_k$, and
 $\gamma_k$ is a point of finite order of $\Pic0 B$. 
In particular $R^j\pi_*(R^if_*(K_X\otimes P_\eta)\otimes P_\beta)$ is a GV sheaf on $B$.
\end{theorem} 
\proof By Theorem \ref{vanishing} and Corollary \ref{higher}(a) 
\begin{equation}\label{intermediate}R^{i+j}(\pi\circ f)_*(K_X\otimes P_\eta\otimes f^*P_\beta)=
\bigoplus_{h+l=i+j}R^h\pi_*(R^lf_*(K_X\otimes P_\eta)\otimes P_\beta)=\bigoplus_k \sigma_k^*(\F_k)\otimes P_{\alpha_k}^{-1}
\end{equation}
We have to prove that the summands of the Chen-Jiang decomposition on the right split in such a way to provide Chen-Jiang decompositions of the individual summands in the middle. This follows from the uniqueness and Fourier-Mukai-theoretic meaning of the Chen-Jiang decomposition (Remarks \ref{CJFM} and \ref{uniqueness}). 
\endproof

\section{Background material: components of  cohomological support loci}  
\subsection{Components of $\mathbf{V^i(K_X)}$} How do components of $V^i(K_X)$ arise? Recalling the notation of the Introduction we have
\begin{equation}\label{fundamental1'}\xymatrix{X\ar[r]^{alb_X}\ar[rd]^f& \mathrm{Alb}\, X\ar[d]^\pi\\
&B\\}
\end{equation}
As we know from (a) and (c) of the Introduction, 
 a component $W$ of $V^i(K_X)$ is of the form  $f^*(\Pic0 B)+{\eta}$, with $\eta$ a point of finite order of $\Pic0 X$. This means that a point $\alpha \in \Pic0 X$  belongs to $W$ if and only if $P_\alpha=P_{ \eta}\otimes f^*P_\beta$ for some $\beta\in \Pic0 B$. Hence, in the notation of \S1,
\[f^*\Pic0B=V^i_f(K_X\otimes P_{\eta})\] 
By  Koll\`ar decomposition (Theorem \ref{decomp}) and projection formula
\[V^i_f(K_X\otimes P_{ \eta})=f^*V^0(R^if_*(K_X\otimes P_{ \eta}))\cup f^* V^1(R^{i-1}f_*(K_X\otimes P_{ \eta}))\cup\cdots \]
By Theorem \ref{vanishing}, all loci in the right hand side  are proper subvarieties of $\Pic0 B$ except for the first one. It follows that
\[W=f^*V^0(R^if_*(K_X\otimes P_{\eta}))\]
By (b) of Theorem \ref{decomp} $R^if_*(K_X\otimes P_{ \eta})$ vanishes for $\dim X-\dim f(X)<i$. This proves   the basic inequality (\ref{fundamental2}):
\[\dim X-\dim f(X)\ge i \>.\]
 Summarizing, so far we got that: \\
\emph{a component $W$ of $V^i(K_X)$ is always of the form 
\[W=f^*V^0(R^if_*(K_X\otimes P_{\eta}))+\eta=f^*\Pic0 B+\eta\]
where $f:X\rightarrow B$ is a morphism to a complex torus such that $\dim X-\dim f(X)\ge i$.}

Next, we consider the Stein factorization of the morphism $f$ of (\ref{fundamental1'})
\begin{equation}\label{fundamental3}\xymatrix{X\ar[r]^{alb_X}\ar[d]^g\ar[rd]^f& \mathrm{Alb}\, X\ar[d]^\pi\\ Y\ar[r]^a
&B\\}
\end{equation}

\begin{lemma}\label{pic} In the above setting $B$ must be $\mathrm{Alb}\,  Y$ and, up to translation,  $a=alb_Y$. If follows that:
all components of $V^i(K_X)$ are  translates of subtori of the form
\[ g^*V^0(
R^ig_*(K_X\otimes P_{\eta}))+\eta=g^*\Pic0 Y+\eta\]
 for  pairs $(g,\eta)$ such that:\\ 
-  $g:X\rightarrow Y$ is an irregular fibration with $\dim X-\dim Y\ge i$ ;\\
-   $\eta$ is a torsion point of $\Pic0 X$ such that $V^0(R^ig_*(K_X\otimes P_{\eta}))=\Pic0 Y$, i.e. $\chi(R^ig_*(K_X\otimes P_{\eta}))>0$.

\noindent Conversely, given a pair $(g,\eta)$ as above, $g^*\Pic0 Y+\eta$ is contained in $V^i(K_X)$.
\end{lemma}

\proof  Since both maps $a$ and $\pi$  factor through $\mathrm{Alb}\, Y$,
 diagram (\ref{fundamental3}) is factorised as follows
 \begin{equation}\label{fundamental4}\xymatrix{X\ar[r]^{alb_X}\ar[d]^g\ar[rd]^{f_Y}\ar@/_3pc/[ddr]_f& \mathrm{Alb}\, X\ar[d]^{\pi_{\mathrm{Alb}\,Y}}\ar@/^3pc/[dd]^\pi\\ Y\ar[r]^{alb_Y}\ar[rd]^a
&\mathrm{Alb}\, Y\ar[d]^\sigma\\
&B\\}
\end{equation}
By Corollary \ref{higher}(c) $R^h{{alb}_Y}_*(R^i{g}_*(K_X\otimes P_{\eta}))=0$ for $h>0$. Therefore $R^i{f_Y}_*(K_X\otimes P_{\eta})={alb_Y}_*R^ig_*(K_X\otimes P_{\eta})$. Hence, by Theorem \ref{vanishing} and an easy Leray spectral sequence $R^ig_*(K_X\otimes P_{\eta})$ is a GV sheaf (with respect to $alb_Y$) and $\chi(R^ig_*(K_X\otimes P_{\eta})=\chi(R^i{f_Y}_*(K_X\otimes P_{\eta}))\ge 0$. We claim that the strict inequality holds, that is: $V^0(R^i {f_Y}_*(K_X\otimes P_{\eta}))=\Pic0 Y$. This implies that $g^*(\Pic0 Y)+\eta$ is contained in $V^i(K_X)$ and contains the component $W$, hence they must be  equal. Moreover $\Pic0B=\Pic0Y$. Therefore the claim proves the Lemma.

To prove what claimed we argue as follows. We know that  $V^0(R^if_*(K_X\otimes P_\eta))=\Pic0 B$.  If $V^0(R^i {f_Y}_*(K_X\otimes P_{\eta}))$ is strictly contained in $\Pic0 Y$ then $\sigma^*\Pic0 B$ must be a component of \break $V^0(R^i{f_Y}_*(K_X\otimes P_{\eta}))$, say of codimension $j$. Then we know by Corollary \ref{el} that $\sigma^*\Pic0 B$ is  also a component of $V^j(R^i{f_Y}_*(K_X\otimes P_{\eta}))=V^j_{alb_Y}(R^ig_*(K_X\otimes P_\eta))$. But, as the map $a$ is finite, $R^i(a\circ g)_*(K_X\otimes P_{\eta}))=a_*R^i{g}_*(K_X\otimes P_{\eta}))$, as above. Therefore, again by an easy Leray spectral sequence, $\Pic0 B=V^j(R^i(a\circ g)_*(K_X\otimes P_{\eta}))$, hence $R^i(a\circ g)_*(K_X\otimes P_{\eta})$ is not a GV sheaf on the complex torus $B$, in contradiction with Theorem \ref{vanishing}.

The last assertion follows by the Koll\'ar decomposition.
\endproof
\subsection{Components of $\mathbf{V^i(R^jg_*(K_X\otimes P_\eta))}$} The previous Lemma relates the loci $V^i(K_X)$ to the loci $V^0(R^if_*(K_X\otimes P_\eta))$ for suitable morphisms to complex tori $f:X\rightarrow B=\mathrm{Alb}\, Y$ or, what is the same, to the loci $V^0(R^ig_*(K_X\otimes P_\eta))$, where in the diagram
\[\xymatrix{X\ar[r]^{alb_X}\ar[d]^g\ar[rd]^f& \mathrm{Alb}\, X\ar[d]^\pi\\ Y\ar[r]^{alb_Y}
&\mathrm{Alb}\, Y\\}\]
$g$ is the Stein factorization of the morphism $f$. More generally, it is  useful to describe in a similar way the components of the cohomological support loci $V^r(R^ig_*(K_X\otimes P_\eta))$, for all $i$ 
and $r$. This is the content of part (a) of the following Lemma. Part (b) provides and explicit description of extremal components.

\begin{lemma}\label{pic2} 
In the above notation, let $\eta$ be a point of finite order of $\Pic0 X$. 

\noindent (a) For all integers $r$ and $i$ the components of $V^r(R^ig_*(K_X\otimes P_\eta))$ 
are of the form
\begin{equation}\label{form} h^*V^0(
R^rh_*(R^ig_*(K_X\otimes P_{\eta})\otimes P_\alpha)+\alpha=h^*\Pic0 Z+\alpha
\end{equation}
 for pairs $(h,\alpha)$ such that:\\ 
-  $h:Y\rightarrow Z$ is an irregular fibration with $\dim Y-\dim Z\ge r$;\\
-   $\alpha$ is a point of finite order of $\Pic0 Y$ such that $V^0(R^rh_*(R^ig_*(K_X\otimes P_{\eta})\otimes P_\alpha))=\Pic0 Z$, i.e. $\chi(R^rh_*(R^ig_*(K_X\otimes P_{\eta})\otimes P_\alpha))>0$.

\noindent Conversely, given a pair $(h,\alpha)$ as above, $h^*\Pic0 Z+\alpha$ is contained in $V^r(R^ig_*(K_X\otimes P_\eta))$. 

\noindent (b) Extremal components for $R^ig_*(K_X\otimes P_\eta))$, i.e. components of $V^r(R^ig_*(K_X\otimes P_\eta))$ of codimension $r$ for some $r$, are of the form \emph{(\ref{form})} for pairs $(h,\alpha)$, where $h:Y\rightarrow Z$ is an irregular fibration such that $\dim Y-\dim Z=r=q(Y)-q(Z)$ \emph{(here $q(Y)$ and $q(Z)$ denote $\dim \mathrm{Alb}\, Y$ and $\dim \mathrm{Alb}\, Z$)} and $\alpha$ is such that $\chi(R^rh_*(R^ig_*(K_X\otimes P_{\eta})\otimes P_\alpha))>0$.

\noindent Conversely, given a pair $(h,\alpha)$ as above, $h^*\Pic0 Z+\alpha$ is \emph{a component}  of  $V^r(R^ig_*(K_X\otimes P_\eta))$ of codimension $r$.

\end{lemma}
\begin{proof} 
We recall that the subtorus Theorem (namely items (a) and (c) of the Introduction) holds as well for the sheaves $R^ig_*(K_X\otimes P_\eta)$ (or, equivalently  for the sheaves $R^if_*(K_X\otimes P_\eta)$). For extremal components this follows at once from
 Theorem \ref{vanishing} and Example \ref{pullback}, but in fact it holds more generally for all components,  see e.g. \cite{hp2} Thm 2.2(b)\footnote{in brief: one can define more generally loci $V^i_{m}(K_X\otimes P_\eta)=\{\alpha\in \Pic0 X\>|\> h^i(K_X\otimes P_\eta\otimes P_\alpha)\ge m\}$ and the Theorems of Green-Lazarsfeld and Simpson-Wang prove as well  that all components $V^i_{m}(K_X\otimes P_\eta)$ are translates of subtori by points of finite order. Then one proves, using Koll\'ar's decomposition, that a component of $V^r(R^i g_*(K_X\otimes P_\eta))$ is also a component of $V^{r+i}_m(K_X\otimes P_\eta)$ for some $m$}. Thus all ingredients for the proof of Lemma \ref{pic} (vanishing theorem, Koll\`ar decomposition, subtorus theorem) hold for the sheaves $R^ig_*(K_X\otimes P_\eta)$ as well, and the argument goes trough without any change, proving (a).
 
 \noindent (b) A component of $V^r(R^ig_*(K_X\otimes P_\eta))$ is extremal if and only if, in the notation of the statement, $ q(Y)- q(Z)=r$. From (a) we have also the basic inequality $\dim Y-\dim Z\ge r$. Since the map $alb_Y$ is finite, also its restriction to  the fibers of $h$ must be finite. Hence the fibers of $h$ surject on the fibers of the homomorphism $\mathrm{Alb}\, Y\rightarrow \mathrm{Alb}\, Z$. Therefore $\dim Y-\dim Z=q(Y)-\dim q(Z)=r$.  Conversely, since $R^ig_*(K_X\otimes P_\eta)$ is a GV sheaf, the codimension of a component of  $V^r(R^ig_*(K_X\otimes P_\eta))$ can't be smaller than $r$. Therefore for every pair $(h,\alpha)$ as in the statement the translated subtorus $h^*\Pic0 Z+\alpha$ is a component of $V^r(R^ig_*(K_X\otimes P_\eta))$, in fact an extremal one. 
 \end{proof}
 
 \section{Background material: comparing Chen-Jiang decompositions}
 
 In view of Lemma \ref{pic2}(b), to study extremal components for sheaves  $R^ig_*(K_X\otimes P_\eta)$ as above, we are led to consider commutative diagrams as follows
 \begin{equation}\label{fundamental6}
 \xymatrix{X\ar[r]^{alb_X}\ar[d]^g\ar[rd]^{f}& \mathrm{Alb}\, X\ar[d]^{\pi_{\mathrm{Alb}\,Y}}\\ Y\ar[r]^{alb_Y}\ar[d]^{h}\ar[rd]^{l}
&\mathrm{Alb}\, Y\ar[d]^{\pi}\\
Z\ar[r]^{a}&B\\} 
\end{equation}
where $X$ is compact K\"ahler, the vertical maps on the right are morphisms of complex tori with connected fibers,  the vertical maps on the left are Stein factorizations, and the lower part of the diagram is such that
\[\dim Y-\dim Z=q(Y)-\dim B:=m \>\]
(from Lemma \ref{pic2}(a) it follows also that $B=\mathrm{Alb}\, Z_B$ and $a_B$ is $alb_{Z_B}$ but we will stick to the notation of (\ref{fundamental8})). Given a point of finite order $\eta\in \Pic0 X$,  by Theorem \ref{vanishing} we have the Chen-Jiang decomposition
\begin{equation}\label{first}R^if_*(K_X\otimes P_\eta)=\bigoplus_k \pi_k^*(\F_k)\otimes P_{\alpha_k}^{-1}
\end{equation}
where each $\pi_k:\mathrm{Alb}\, Yj\rightarrow B_k$ is a surjective morphism of complex tori with connected fibers. Moreover, given a point of finite order $\alpha\in \Pic0 Y$, by Theorem \ref{casino}, we have the Chen-Jiang decomposition
\begin{equation}\label{second}R^m\pi_*(R^if_*(K_X\otimes P_\eta)\otimes P_\alpha)=
\bigoplus_h \sigma_h^*(\G_h)\otimes P_{\gamma_h}^{-1}
\end{equation}
 where  $\sigma_h:B\rightarrow C_k$  are surjective morphisms of complex tori with connected fibers. We would like to compare the two decompositions. This is the content
 of the following useful property, due to Chen and Jiang (\cite{chen-jiang} Lemma 3.7).

 \begin{lemma}\label{final?} We adopt the notation above. Then 
 \begin{equation}\label{third}
 R^m\pi_*(R^if_*(K_X\otimes P_\eta)\otimes P_\alpha)=\bigoplus_{\pi_k^*\Pic0 B_k+\alpha_k\subset \pi^*\Pic0 B+\alpha} R^m\pi_*(\pi_k^*(\F_k)\otimes P_{\alpha_k}^{-1}\otimes P_\alpha)
 \end{equation}
 and (\ref{third}) coincides with the Chen-Jiang decomposition (\ref{second}) for $j=m$. In particular, for each summand $\pi_k^*\F_k\otimes P_{\alpha_k}^{-1}$ in (\ref{first}), we have that  $R^m\pi_*(\pi_k^*\F_k\otimes P_{\alpha_k}^{-1}\otimes P_\alpha)\ne 0$ if and only if \break $\pi_k^*\Pic0 B_k+\alpha_k\subset \pi^*\Pic0 B+\alpha$. If this is the case, up to twist with a line bundle in $\pi^*\Pic0 B$, $\pi_k^*\F_k=\pi^*\sigma_h^*\G_h$ for a certain Chen-Jiang summand $\sigma_h^*\G_h$ of (\ref{second}). 
 \end{lemma}
 \begin{proof}  Let us denote $\mathcal H=R^if_*(K_X\otimes P_\eta)$. We know that each factor of the Chen-Jiang decomposition (\ref{second}) corresponds to a component of $V^u(R^m\pi_*(\H\otimes P_\alpha))$ of codimension $u$ in $\Pic0 B$ (including possibly $u=0$). Since $\dim Y-\dim Z=q(Y)-\dim B=m$, by Koll\'ar decomposition such components induce components of $V^{m+u}(\H\otimes P_\alpha)$ of codimension $m+u$ in $\Pic0 Y$, i.e. extremal components for $\H$.  By Remark \ref{important}(b), this means that $\pi^*(\cdot)\otimes P_\alpha^{-1}$ of the Chen-Jiang decomposition (\ref{second})  is part of the Chen-Jiang decomposition (\ref{first}). With this in mind, let us apply $R^m\pi_*((\cdot)\otimes P_\alpha)$ to (\ref{first}). We get  another Chen-Jiang decomposition  of $R^m\pi_*(\G\otimes P_\alpha)$, essentially coinciding with (\ref{second}) (Remark \ref{uniqueness}). Applying $\pi^*(\cdot)\otimes P_\alpha^{-1}$ we obtain a Chen-Jiang decomposition  of a part of (\ref{first}). By uniqueness of Chen-Jiang decompositions (Remark \ref{uniqueness} again), this essentialy coincides with the right hand side of  (\ref{third}). Namely it cannot happen that $R^m\pi_*(\pi_k^*\F_k\otimes P_{\alpha_k}^{-1}\otimes P_\alpha)\ne 0$ if $\pi_k^*\Pic0 B_k+\alpha_k$ is not  contained in $\pi^*\Pic0 B+\alpha$.
 \end{proof}

\section{Proof of Theorem \ref{2}}

As it is technically easier, we prove directly  Theorem \ref{2}, before proving Theorem \ref{1}.

\proof  (a) Let   $g:X\rightarrow Y$ be an irregular fibration with 
\[m:=\dim X-\dim Y>p.\] For all $\eta\in \Pic0 (g)$ the coherent sheaf $R^{m}g_*(K_X\otimes P_\eta)$ is a non-zero and GV. Therefore \break $V^0(R^{m}g_*(K_X\otimes P_\eta))$ is non-empty by non-vanishing (Theorem \ref{non-vanishing}). By Koll\'ar decomposition (Theorem \ref{decomp}) and projection formula we have that
 if $V^0(Y, R^{m}g_*(K_X\otimes P_\eta))$ was positive-dimensional then also $V^{m}(X,K_X\otimes P_\eta)$ would be positive-dimensional, against the definition of $p$. 
 Therefore we are left with the case when $V^0(Y, R^{m}g_*(K_X\otimes P_\eta))$ is zero-dimensional and the statement of (a) will proved as soon as we show that in this case: (i) $Y$ is bimeromorphic to a complex torus, and (ii) $\Pic0 Y=g^*\Pic0 Y$. 
 We prove (ii)  first. 
 By Corollary \ref{el}(a), a $0$-dimensional component of $V^0(R^{m}g_*(K_X\otimes P_\eta))$ is also a    component of $V^{q(Y)}(R^{d(g)}g_*(K_X\otimes P_\eta))$. Again by Koll\'ar decomposition this induces a component of $V^{m+q(Y)}(K_X)$, which is impossible unless $m+q(Y)\le\dim X$. But $m=\dim X-\dim Y$ and $\dim Y\le q(Y)$. Therefore $\dim Y=q(Y)$ and, in conclusion, $m+q(Y)=\dim X$. But, since $V^{\dim X}(K_X)=\{\hat 0\}$, the above is possible only when $\eta\in g^*\Pic0Y$. This proves (ii).

 It remains to  prove (i).  Since the loci $V^i(K_X)$ are bimeromorphic invariants of compact K\"ahler manifolds, after desingularizing $Y$ and replacing $X$ with a suitable bimeromorphic compact Kahler manifold, we can assume that
$Y$ is a compact complex manifold. We can assume also  that $Y$ is K\"ahler (e.g. \cite{catanese2} 1.9). Therefore $R^{m}g_*K_X=K_Y$ (\cite{kollar1} Prop. 7.6. See also  \cite{takegoshi} Th. 6.10(iii) for the analytic setting)  and, by the above, $V^0(K_Y)$ is $0$-dimensional.  But a well known result of Ein-Lazarsfeld (\cite{chen-hacon1}), tells that  this is the case if and only if $Y$ is bimeromorphic to a complex torus. For the reader's convenience, we outline the proof: if $V^0(K_Y)$ is $0$-dimensional then, by Remark \ref{important}, the complex tori $B_k$ corresponding to the Chen-Jiang summands of  ${alb_Y}_*(K_Y)$ are $0$-dimensional. Therefore ${alb_Y}_*(K_Y)$ would be the direct sum of toplogically trivial line bundles on $\mathrm{Alb}\,Y$. But, since $V^{q(Y)}(Y, K_Y)=\{\hat 0\}$ (recall that $q(Y)=\dim Y$), it must be ${alb_Y}_*K_Y=\OO_{\mathrm{Alb}\, Y}$. Hence $alb_Y$ is  bimeromorphic. 

\noindent (b)  Thanks to Lemma \ref{pic} and (a) all positive-dimensional components of $V^p(K_X)$  are translates of subtori of the form $g^*V_0(R^p g_*(K_X\otimes P_\eta))=g^*\Pic0 Y$ for irregular fibrations $g:X\rightarrow Y$ with $\dim X-\dim Y=p$. Therefore the restriction of $P_\eta$ to a general fiber of $g$ must be trivial, that is $\eta$ belongs to $\Pic0 (g)$. It remains to prove that all components  of $\Pic0 (g)$ (respectively: all components of $\Pic0 (g)$ but $g^*\Pic0 Y$ if $Y$ is bimeromorphic to a complex torus)  are as above. 
One proceeds  as in (a). To begin with, we claim that for every irregular fibration $g$ as above, every component  of $\Pic0 (g)-g^*\Pic0 Y$  is contained in $V^p(K_X)$ (hence, as it is easy to see, it is a \emph{component} of $V^p(K_X)$). This follows from Lemma \ref{pic} and (a) as well, because in any case $V^0(R^pg_*(K_X\otimes P_\eta))$ is non-empty (Lemma \ref{non-vanishing}). Thus, as in the proof of (a), if $P_\eta$ does not belong to $g^*\Pic0 Y$ then $V^0(R^pg_*(K_X\otimes P_\eta))$ is positive-dimensional. But a component of codimension $j$ of $V^0(R^pg_*(K_X\otimes P_\eta))$ is also a component of 
$V^j(R^pg_*(K_X\otimes P_\eta))$ (Corollary \ref{el}). Therefore, by Koll\'ar's decomposition, it induces a positive-dimensional component of $V^{p+j}(K_X)$, which contradicts the definition of $p$. This proves what claimed. By the same reason, either $V^0(R^pg_*(K_X))$ is the full $\Pic0 Y$ or it is zero-dimensional. In the former case also $g^*\Pic0 Y$ is a component of $V^p(K_X)$. In the latter case, as in (a), $Y$ must bimeromorphic to a complex torus.
\endproof

\section{the standard part}

In this section we will prove Theorem \ref{1} and its Corollaries \ref{2mezzo} and \ref{3}.

\begin{proof} (of Theorem \ref{1})  Standard pairs $(W,i)$ arise from irregular fibrations $g: X\rightarrow Y$ 
  with 
 \[\dim X-\dim Y=i\]
  They sit in the usual diagram
  \begin{equation}\label{fundamental7}
 \xymatrix{X\ar[r]^{alb_X}\ar[d]^g\ar[rd]^{f}& \mathrm{Alb}\, X\ar[d]^{\pi_{\mathrm{Alb}\,Y}}\\ Y\ar[r]^{alb_Y}
&\mathrm{Alb}\, Y}
\end{equation}
 We claim that the components $W$ are of the form 
 \[W= g^*V^0(R^ig_*(K_X\otimes P_\eta))+\eta=g^*\Pic0 Y+\eta\]
  for \emph{all} line bundles $P_\eta$ such that $V^0(R^ig_*(K_X\otimes P_\eta))=\Pic0 Y$. As we know, this means that $\chi(R^ig_*(K_X\otimes P_\eta))>0$ or, what is the same,
  \begin{equation}\label{non-zero}
 \chi(R^if_*(K_X\otimes P_\eta))>0 \>.
 \end{equation}
 Indeed  the components $W$ are as above by Lemma \ref{pic}. Conversely, by the same Lemma,  $g^*\Pic0 Y+\eta$ is contained in $V^i(K_X)$ for such line bundles $P_\eta$ and it is in fact a component,  otherwise it would be strictly contained in a component, say $U$, of $V^i(K_X)$. But by Lemma \ref{pic} the  component $U$ would arise from another  irregular fibration $g^\prime: X\rightarrow Y^\prime$ factoring $g$, and this is impossible since $\dim X-\dim Y^\prime$ would be strictly smaller than $i=\dim X-\dim Y$, against Lemma \ref{pic}.

  Furthermore, since $i=\dim X-\dim Y$, $\eta\in\Pic0 (g)$  if and only  the GV sheaf $R^ig_*(K_X\otimes P_\eta)$ is non-zero, equivalently $R^if_*(K_X\otimes P_\eta)$ is  non-zero. However it can happen that 
\begin{equation}\label{zero} \chi(R^if_*(K_X\otimes P_\eta))=0
\end{equation}
Therefore, recalling the remark preceding the statement of Theorem \ref{1}, our task consists precisely in describing the subset $N_g$
whose elements are classes $[\eta]\in \Pic0 (g)/g^*\Pic0 Y$ such that 
(\ref{zero}) holds. Theorem \ref{1} is equivalent to the following
\begin{claim}\label{group} $N_g$ is either empty or a subgroup of $\Pic0 (g)/g^*\Pic0 Y$.  The latter case happens if and only if $\chi(K_{\widetilde Y})=0$, where $\widetilde Y$ is a desingularization of $Y$. If furthermore the Albanese morphism of $Y$ is surjective (i.e. $\dim Y=q(Y)$) then $N_g=\{0\}$.
\end{claim}

 As in the proof of Theorem \ref{2}, we remark that, since the loci $V^i(K_X)$ are bimeromorphic invariants, we can replace the fibration $g:X\rightarrow Y$ with a fibration $\tilde g: X^\prime\rightarrow \widetilde Y$, where $\widetilde Y\rightarrow Y$ is a K\"ahler desingularization and $X^\prime$ is K\"ahler and bimeromorphic to $X$. Since  $R^i{\tilde g}_*(K_{X^\prime})=K_{\widetilde Y}$ (\cite{kollar1} Prop. 7.6,  \cite{takegoshi} Th. 6.10(iii)) the second sentence of  Claim \ref{group} follows from the first one.

  The first and third sentences of   Claim \ref{group} are proved following the ideas of Chen and Jiang \cite{chen-jiang}.
 To this purpose, we consider the finite set of all proper (i.e. strictly contained) subtori of $\Pic0 Y$ of the form 
\[T=\pi_{B}^*\Pic0 B\] where $\pi_{B}:\mathrm{alb}\, Y \rightarrow  B$ is a surjective homomorphism with connected fibres, such that some translates of $T$ are extremal components of $R^if_*(K_X\otimes P_\eta)$ for some $\eta\in\Pic0 (g)$.  Equivalently, they are all proper subtori of $\Pic0 Y$ which are dual to quotients $\mathrm{Alb}\, Y\rightarrow B_k$ appearing in the Chen-Jiang decomposition of $R^if_*(K_X\otimes P_\eta)$ for some $\eta \in\Pic0 (g)$.  For easy reference, we will call  them  \emph{extremal subtori} of $\Pic0 Y$\footnote{we recall (see  the footnote to Remark \ref{uniqueness}) that, unlike the homomorphisms $\pi_{B}$, the subtori $T$ are  uniquely determined by the Chen-Jiang decomposition}.  We consider also the usual diagram with the Stein factorizations 
 \begin{equation}\label{fundamental8}
\xymatrix{X\ar[r]^{alb_X}\ar[d]^g\ar[rd]^f&\mathrm{Alb}\,X\ar[d]^{\pi_{\mathrm{Alb}\,Y}}\\Y\ar[r]^{alb_Y}\ar[d]^{h_{B}}\ar[rd]^{l_{B}}&{Alb}\, Y\ar[d]^{\pi_{B}}\\ Z_{B}\ar[r]^{a_{B}}&B}
\end{equation}

Given a pair $(\eta,T)$, where $\eta\in \Pic0 (g)$ and $T$ is an extremal subtorus of $\Pic0 Y$, we have $\dim Y-\dim Z_{B}=q(Y)-\dim B:=m_B$ (Corollary \ref{el} and Lemma \ref{pic2}(b)). We consider the finite set, depending on $[\eta]\in\Pic0 X/g^*\Pic0 Y=\Gamma_g$ and $T$,
 \[\mathcal{N}_{T}([\eta])=\{[\alpha]\in {\Pic0 (Y)}/T \>\> \>|\>\> R^{m_{B}}{h_{B}}_*(R^ig_*(K_X\otimes P_\eta)\otimes P_\alpha)\ne 0\}\] 
 We have that:

 \noindent\emph{(a) the set of $[\eta]\in \Gamma_g$ such that $\mathcal{N}_{T}([\eta])$ is non-empty is a subgroup \ $\Sigma_{T,g}\le \Gamma_g $ .}\\
  This is because the condition 
\[R^{m_B}{h_{B}}_*(R^ig_*(K_X\otimes P_\eta)\otimes P_\alpha)\ne 0\]
 is equivalent, by projection formula,  to \[\eta+g^*(\alpha)\in
 \Pic0 (h_{B}\circ g)\] 
 i.e. $\eta\in \Pic0(g)\cap (\Pic0(h_B\circ g)+g^*\Pic0Y)$.
 
 \noindent\emph{(b) for $\eta\in \Sigma_{T,g}$ the set $\mathcal{N}_T([\eta])$ is in bijection with the finite group $\Pic0(h_B)/T$.}\\
  To prove this, note that the group $\Pic0(h_B\circ g)$ sits in an extension as follows
 \[0\rightarrow g^*\Pic0(h)\rightarrow \Pic0(h_B\circ g)\rightarrow \Gamma^\prime_{h_B,g}\rightarrow 0\]
 and
 \begin{equation}\label{aggiunta}\Pic0(h_B\circ g)\cap g^*\Pic0(Y)=g^*\Pic0(h_B) \>.
 \end{equation}
 \emph{(b)} follows immediately from (\ref{aggiunta}) because, given two elements $[\alpha]$ and $[\beta]$ in $\mathcal{N}_{T}([\eta]) $,   $g^*(\alpha-\beta)\in \Pic0(h_B\circ g)\cap g^*\Pic0 Y$.
 \begin{claim}\label{Ng} If $N_g$ is not empty, then it is the intersection of the subgroups $\Sigma_{T,g}$ for all maximal \emph{(with respect to inclusion)} extremal subtori $T$ of $\Pic0 Y$.
 \end{claim}  
 This proves  Claim \ref{group}, hence the Theorem. Indeed if $alb_Y$ is surjective (equivalently, $\dim Y-q(Y)$) then the trivial subtorus, denoted $\widehat 0$, is a component of $V^{q(Y)}(Y, K_Y)=V^{q(Y)}(\mathrm{Alb}\, Y, alb_{Y,*}(K_Y))$, hence an extremal component of $K_Y=R^if_*(K_X) $. Therefore $\widehat 0$ is an extremal subtorus of $\Pic0 Y$, and it is clear that $\Sigma_{\widehat 0,g}$ is zero.
 \proof (of Claim \ref{Ng}) 
 We recall that for $\eta\in \Pic0(g)$ the (generic) rank of $R^ig_*(K_X\otimes P_\eta)$ is equal to one. Hence the  rank of $R^if_*(K_X\otimes P_\eta)$ is equal to $\deg alb_Y$. In turn $R^if_*(K_X\otimes P_\eta)$ is  the direct sum its Chen-Jiang summands. These are of two types:  the M-regular one and the other summands, arising, according to Lemma \ref{pic2}(b)
  from proper extremal subtori $T=\pi_B^*\Pic0 B$ (see also Remark \ref{important}, Corollary \ref{el} and Lemma \ref{pic2}). Therefore, for each $[\eta]\in \Gamma_g=\Pic0 (g)/g^*\Pic0 Y$
 \begin{equation}\label{tauto}\deg alb_Y=H([\eta])+K([\eta])
 \end{equation}
 where $H$ and $K$  denote respectively the rank of the M-regular factor and the sum of the ranks of the other factors. 
 We assert that:
 
 \noindent\emph{(c) The integer $K([\eta])$ is maximal if and only if $[\eta]$ belongs to the to the intersection of the subgroups $\Sigma_{T,g}$ for all proper extremal subtori $T$. }\\
  To prove (c) we note that, by Lemma \ref{final?}, every non-M-regular Chen-Jiang factor of $R^if_*(K_X\otimes P_\eta)$ is the pullback of a Chen-Jiang summand of 
$R^{m_B}{\pi_B}_*(R^if_*(K_X\otimes P_\eta)\otimes P_\alpha)$ for some  proper extremal subtorus $T=\pi_B^*\Pic0 B$ and $[\alpha]\in N_T([\eta])$. If this is the case, we say that such Chen-Jiang summand \emph{belongs to $T$} (note that, according to this terminology,  a Chen-Jiang factor can belong to more than one extremal subtorus). Again by Lemma \ref{final?} the part of the Chen-Jiang decomposition of $R^if_*(K_X\otimes P_\eta)$ belonging to $T$ is the Chen-Jiang decomposition of
\begin{equation}\label{tento}\bigoplus_{[\alpha]\in N_T([\eta])}\pi_B^*(R^{m_B}{\pi_B}_*(R^if_*(K_X\otimes P_\eta)\otimes P_\alpha))\otimes P_\alpha^{-1}
\end{equation} 
 Since the  rank of 
 \begin{equation}\label{provo}R^{m_{B}}{h_{B}}_*(R^ig_*(K_X\otimes P_\eta)\otimes P_\alpha)=R^{m_B+i}(h_B\circ g)_*(K_X\otimes g^*\alpha)
 \end{equation}
  is equal to one, 
the  rank of each summand in (\ref{tento}) is $\deg a_B$ (see the notation of (\ref{fundamental8})), hence it doesn't depend on $[\alpha]\in N_T([\eta])$.

  The above, together with \emph{(b)}, shows that the rank of the part  of the Chen-Jiang decomposition of $R^if_*(K_X\otimes P_\eta)$ belonging to $T$ is equal to the integer (independent on $[\eta]) $) \[\deg a_B\cdot |\Pic0(h_B)/T|\] if $[\eta]\in \Sigma_{T,g}$,  and zero otherwise.  This shows that the integer $K(([\eta])$ is maximal if and only if $[\eta]$ is in the intersection of the subgroups $\Sigma_{T,g}$ for all extremal subtori $T$. Indeed, if this is the case,   for all extremal subtori $T$  the  part of the Chen-Jiang decomposition of  $R^if_*(K_X\otimes P_\eta)$ belonging to $T$ is non-zero, of rank as in (\ref{provo}). Conversely, if $[\eta]\not\in\Sigma_{T,g}$ for some extremal subtorus $T$ the part of the Chen-Jiang decomposition  of  $R^if_*(K_X\otimes P_\eta)$ belonging to $T$ vanishes. A little argument with Lemma \ref{final?} shows that $K([\eta])$ can't be maximal. This proves \emph{(c}).
  
  Finally, assume that the group $N_g$ is non-empty. This means that $H([\eta])=0$ for some $[\eta]\in \Gamma_g=\Pic0 (g)/g^*\Pic0 Y$. By (\ref{tauto}), this is equivalent to the fact $K([\eta])$ is maximal, namely equal to $\deg alb_Y$. Then Claim \ref{Ng} follows from (c). 
 \end{proof}
Corollary \ref{2mezzo} of the Introduction coincides with Claim \ref{group}. 

Corollary \ref{3}.   Since we are assuming $\dim alb_X(X)=\dim X$ we have that $R^ialb_{X*}(K_X)=0$ for $i>0$. Therefore $V^i(X,K_X)=V^i(\mathrm{Alb}\, X, alb_{X *}K_X)$ (hence, as it is well known from the Theorems of Green and Lazarsfeld, $K_X$ is GV). Therefore every component of $V^0(alb_{X*}K_X)$ is an extremal component of some $V^j(alb_{X*}K_X)$, hence it is standard  (Lemma \ref{pic2}, ). Conversely,   since $V^0(alb_{X*}K_X)$ is strictly contained in $\Pic0 X$, the $M$-regular summand of the Chen-Jiang decomposition of ${alb_X}_*K_X$ is absent. Therefore, by Remark \ref{important}, every extremal component for $K_X$ is also a component of $V^0(K_X)$.  Therefore Theorem \ref{1} applies.

\section{Compact K\"ahler manifolds with $q(X)=\dim X$, $p_1(X)\ne 0$ and $p_3(X)=2$}

This section is devoted to the proof of Theorem \ref{4}, which is a slightly weaker extension to the compact K\"ahler setting  of a result of Hacon-Pardini  in the algebraic case (\cite{hacon-pardini} Th.6.1).   In this sort of matters a mayor role is played by
multiplication maps
\begin{equation}\label{mult}\bigoplus_{\eta\in W} H^0(X, L\otimes P_\eta)\otimes H^0(X, M\otimes P_\eta^{-1})\rightarrow H^0(X,L\otimes M)
\end{equation}
where $W$ is a suitable subvariety of $\Pic0 X$ (usually the translate of a subtorus). It is clear that if $ h^0(X, L\otimes P_\eta)=k$ and $h^0( M\otimes P_\eta^{-1})=h$ generically on $W$ then 
\begin{equation}\label{newversion4}H^0(X,L\otimes M)\ge \dim W+k+h-1
\end{equation}

Let $X$ be a compact K\"ahler manifold with $\dim(X)=q(X)$, $p_1(X)\ne 0$ and $p_3(X)=2$.  Note that this implies that $p_1(X)=1$ $p_2(X)=2$. Indeed $p_1(X)=1$ since otherwise, by a repeated use of the map (\ref{mult}) for $W=\{\hat 0\}$, it follows that $p_3(X)>3$. Moreover $p_2(X)>1$ since otherwise, by \cite{pps} Thm 1 (previously \cite{chen-hacon1} in the algebraic case)   $X$  would be bimeromorphic to a complex torus. Therefore $p_2(X)=2$.

\begin{step} (i) The Albanese map of $X$ is surjective \emph{(hence, in view of the hypothesis $\dim X=q(X)$, generically finite onto $\mathrm{Alb}\, X$)}. \\
\noindent (ii) $X$ has fibrations onto elliptic curves $g_i:X\rightarrow E_i$  such that $\Pic0(g_i)=g_i^*\Pic0E_i\cup( g_i^*\Pic0 E_i+\eta_i)$ \  for $i=1,... ,k$  \emph{(hence the points  $\eta_i$ are of order two)} and 
\[V^0(K_X)=\{\hat 0\}\cup \bigcup_{i=1,.. ,k}(g_i^*\Pic0 E_i+\eta_i)\] 
Morover $h^0(K_X\otimes P_\alpha)= 1$ generically on $g_i^*\Pic0 E_k+\eta_i$ for all $i=1,\dots , k$,.
\end{step}
\begin{proof} In the first place we claim that the origin $\hat 0$ is an isolated point of $V^0(X)$. As it is well known by \cite{el}, this implies that the Albanese map of $X$ is surjective (proof: since $V^0(X,K_X)=V^0(\mathrm{Alb}\, X,{alb_X}_*K_X)$, the Albanese map of $X$ is surjective by Corollary \ref{el}(a)). To prove what claimed, we observe that, for a positive-dimensional subtorus $W\subset V^0(K_X)$,  (\ref{mult}) and (\ref{newversion4}) for $L=K_X\otimes P_\alpha$, with $\alpha\in W$ and  $M=K_X$ would imply that $h^0(K_X^2\otimes P_\alpha)\ge 2$ for all $\alpha\in W$. Then  (\ref{mult}) and (\ref{newversion4}) for $L=K^2_X$ and $M=K_X$ and $W$ as above would imply $p_3(X)>2$.

Now  $V^0(K_X)$ is invariant with respect to the natural involution of $\Pic0 X$ (by Cor. \ref{3}, but this was already proved in \cite{chen-jiang} Thm 3.5  and \cite{pps}, Cor. 16.2). Therefore, given a positive dimensional component $W$ of $V^0(K_X)$, we can consider the map (\ref{mult}) with $L=M=K_X$. Since $p_2(X)=2$ we get from (\ref{newversion4})  that $\dim W=1$ and that $h^0(K_X\otimes P_\alpha)=1$ generically on $W$. Therefore each of these components is a non-trivial translate of an elliptic curve $\widehat{E_i}$. Dualizing we get maps $g_i: X\rightarrow E_i$:
\[\xymatrix{X\ar[r]^{alb_X}\ar[rd]^{g_i}&\mathrm{Alb}\, X\ar[d]^{\pi_i}\\ &E_i}\]
 Next, we prove that the map $g_i$ has connected fibers. In fact  $\dim W$ is equal to the genus of the Stein factorization of $g_i$, say $C_i$. Therefore $g(C_i)=1$. Hence $g_i$ has already connected fibers since otherwise it would factorize trough an \'etale map. This would imply that   $alb_X$  factorize through an  \'etale map, which is impossible.

Finally, the fact that $\Pic0 (g_i)$ has only two components -- or equivalently (by Beauville's theorem mentioned in the Introduction and generalized by Corollary \ref{2})  there is exactly one component of $V^{\dim X-1}(K_X)$ for each $g_i$ --  is as follows. Let $[\eta^1_i],[\eta^2_i]\in \Pic0 (g_i)/g_i^*\Pic0 E_i$.  We consider the following realization of the map  (\ref{mult}) 
\[\bigoplus_{\eta\in \Pic0 E^i} H^0(X, K_X\otimes P_{\eta_i^1}\otimes P_\alpha\otimes P_\eta)\otimes H^0(X, K_X\otimes P_{\eta_i^2} \otimes P_\eta^{-1})\rightarrow H^0(X,K_X^2\otimes P_{\eta_i^1}\otimes P_{\eta_i^2}\otimes P_\alpha),\]
 for any  $\alpha\in g_i^*\Pic0 E_i$. By (\ref{newversion4}) this would imply that $H^0(K_X^2\otimes P_\beta)\ge 2$ for all $\beta\in g_i^*\Pic0 E_i+\eta_i^1+\eta_i^2$. Moreover, again by Beauville's theorem, also $g_i^*\Pic0 E_i-\eta_i^1-\eta_i^2$ would be  a component of $V^0(K_X)$ (unless $[\eta^1_i]=-[\eta^2_i]$). But then the map (\ref{mult})  for $L=K_X^2$ $M=K_X$ and $W=g_i^*\Pic0 E_i+\eta_1+\eta_2$ would imply, via the inequality  (\ref{newversion4}), that $p_3(X)\ge 3$. 
 \end{proof} 

\begin{step} Keeping the previous notation, $k=1$ and ${alb_X}_*(K_X)=\OO_{\mathrm{Alb}\, X}\oplus ( \OO_E(p)\otimes P_\eta)$, where $p$ is a point of $E$.
\end{step}
\begin{proof} Each component of $V^0(K_X)$ corresponds to a factor of Chen-Jiang decomposition of ${alb_X}_*K_X$:
\[ alb_*K_X=\OO_{\mathrm{Alb}\,X}\oplus\bigoplus_{i=1,..\, ,k}(\pi_i^*\F_i)\otimes P_{\eta_i}\]
Since the maps $g_i:X\rightarrow E_i$ have already connected fibers, we know from Lemma \ref{final?} that the $\F_j$'s (M-regular sheaves on the elliptic curves $E_i$) are the pullback of the $M$-regular summands of the Chen-Jiang decomposition of $R^{\dim X-1}{g_i}_*(K_X\otimes P_{\eta_i})$. However, since such M-regular sheaves have generic rank equal to one, and they are torsion-free, they must be line bundles of positive degree, one for each elliptic curve $E_i$. 
Since $h^0(X, g_i^*(L_i\otimes P_\alpha)\otimes P_{\eta_i})=1$ for general $\alpha\in \Pic0 E_i$, it follows that $ \F_i=\OO_{E_i}(p_i)$, for a point $p_i$ on  $E_i$. Moreover, since $H^0(K_X^2)=2$ the map
\[\bigcup_{\alpha\in\Pic0 E_i}H^0(X, g_i^*(\OO_{E_i}(p_i)\otimes P_\alpha)\otimes P_{\eta_i} )\otimes H^0(X,g_i^*(\OO_{E_i}(p_i)\otimes P_\alpha^{-1})\otimes P_{\eta_i})\rightarrow H^0(X,K_X^2)\]
is surjective.  It follows that $H^0(X,K_X^2)=H^0(X,g_i^*\OO_{E_i}(2p_i))$ and the bicanonical map (in fact morphism) of $X$ factors through  $E_i$. Therefore there is only one such elliptic curve $E_i$.
\end{proof}

\providecommand{\bysame}{\leavevmode\hbox
to3em{\hrulefill}\thinspace}

\end{document}